\documentclass[letterpaper, 10 pt, conference]{ieeeconf}  

\IEEEoverridecommandlockouts  

\usepackage{cite}
\usepackage{amsmath,amssymb,amsfonts}

\usepackage{amsthm}

\newtheorem{theorem}{Theorem}

\newtheorem*{remark*}{Remark}

\newtheorem{definition}{Definition}

\newtheorem*{conjecture*}{Conjecture}
\newtheorem{prop}{Proposition}
\newtheorem{example}{Example}
\newtheorem*{problem*}{Problem}
\usepackage[linesnumbered]{algorithm2e}
\usepackage{graphicx}
\usepackage{textcomp}
\usepackage{xcolor}
\usepackage{cleveref}

\title{\LARGE \bf
Global Sensitivity Analysis for the Linear Assignment Problem}

\author{Elad Michael, Tony A. Wood, Chris Manzie, and Iman Shames
\thanks{$^{*}$All authors are with Department of Electrical and Electronic Engineering at the University of Melbourne
{\tt\small eladm@student.unimelb.edu.au,
        \{wood.t,manziec,iman.shames\}@unimelb.edu.au}}%
}

\begin{document}

\maketitle
\thispagestyle{empty}
\pagestyle{empty}

\begin{abstract}
In this paper, the following question is addressed: given a linear assignment problem, how much can the all of the individual assignment weights be perturbed without changing the optimal assignment? The extension of results involving perturbations in just one edge or one row/column are presented. Algorithms for the derivation of these bounds are provided. We also show how these bounds may be used to prevent assignment churning in a multi-vehicle guidance scenario.
\end{abstract}

\section{Introduction}

Task assignment is a fundamental part of multi-agent control. Here we represent task assignment optimization over a weighted bipartite graph, with agents and tasks as vertices and assignments as edges. The objective function may be minimum completion time of all tasks, minimum total fuel usage, balanced completion times, or any of a multitude of other choices\cite{pentico2007assignment}. For many applications, the edge weights that are parameters in the objective are uncertain. Uncertainties may be due to sensor measurement errors, finite difference estimation, or quantization in calculation or communication. By studying which perturbations the optimal assignment is invariant to, we characterize the uncertainty in edge weights that can be tolerated, and when the optimal assignment with the noisy measurements is optimal with respect to the ground truth.

The terms \emph{robustness} of the assignment will be used frequently throughout this paper. Similar analyses may use terms such as ``sensitivity"\cite{lin2007sensitivity}, ``stability"\cite{klaus2013paths}, or ``post-optimality"\cite{KINDERVATER198576} of optimisation problems. These all explore how the set of optimizers or optimal cost change as the problem parameters are varied. The different definitions of sensitivity analysis are covered in \cite{jansen1997sensitivity}, as well as some practical applications of sensitivity analysis for linear programs. Due to the additional constraints and particular structure of assignments problems compared to a standard linear program, we focus on which perturbations to the problem parameters the optimal assignment is invariant to, ``Type 2 Sensitivity Analysis"\cite{jansen1997sensitivity}. The robustness or sensitivity of the assignment problem in this paper refers to intervals within which the edge weights can vary without changing the optimal assignment.

We focus on the robustness of the linear assignment problem (LAP). The LAP minimizes the sum of all edge weights in an assignment, i.e. finds the matching of agents to tasks which minimizes the total cost. The linear assignment problem has many and diverse applications. For example, finding maximum flow within a network\cite{johnsenlr}, lower bounding solutions to the quadratic assignment problem\cite{adams1994improved}, and maximum a-posteriori data association for object tracking\cite{danford2007joint}. 

In this paper we focus on understanding which perturbations the optimal assignment is robust to. However, in some applications, it is preferable to find a potentially sub-optimal assignment which is less sensitive to perturbation. For linear assignment, \cite{bertsimas2004price} provides an algorithm with guaranteed robustness to uncertainty in a subset of the parameters. Similarly, but for the bottleneck assignment problem, \cite{volgenant2010improved} provides robustness to weights within predefined intervals, as well as complexity improvements on similar bottleneck assignment algorithms. However, these papers necessarily sacrifice optimality for robustness. We focus on applications that involve quantifying the robustness of the optimal assignment.

Previous work in assignment problem sensitivity has been restricted to perturbations in one or a small subset of edges. The sensitivity of the linear assignment problem to perturbations in a \emph{single} edge weight perturbation is covered by \cite{lin2003sensitivity} and \cite{volgenant2006addendum}. Robustness to coupled perturbations in the weight of all edges incident on a vertex (one "row" or "column" of edge weights) is investigated in~\cite{lin2007sensitivity} and~\cite{liu2011assessing}. In a multi-agent scenario, with weights as distances for example, this could be interpreted as an error in the state of an agent or a task, causing a coupled error in the weights of all edges incident on that agent or task. Most relevant to the work presented here for the LAP, \cite{sotskov1995some} examines robustness of all weights under \emph{uniform} perturbation, and proves complexity results. A similar analysis to the one provided is presented in \cite{michael_2019}, however for the bottleneck assignment problem.

The contribution of this work is to allow perturbations in \emph{all} edge weights, without coupling. That is, each weight may be perturbed individually and independently. This coincides with the work done in~\cite{sotskov1995some} for the minimal sensitivity edge. However, the results of \cite{sotskov1995some} are conservative for all other edges. In other work, such as \cite{lin2007sensitivity}\cite{liu2011assessing}, simultaneous perturbation is only considered in edges sharing a common vertex. These cannot be co-assigned due to the assignment constraints, and thus extend the single edge results to considering a subset of edges. However, in the case of the linear assignment problem, the changing of an edge weight effects all assignments, in terms of relative cost. We provide algorithms which compute the sensitivity for a given weighted bipartite graph, and discuss the computational complexities. With these algorithms, we formulate a set of sufficient conditions with which an assignment made on noisy measurements can be proved to be optimal despite the noise. The results are used to mitigate the effects of ``assignment churning''\cite{bertuccelli2009real} in a multi-vehicle guidance application.
%
%
The paper is organized as follows. Section~\ref{sec:probForm} is devoted to background on the assignment problem and a formal definition of robustness. Section~\ref{sec:intervals} shows that a single allowable perturbation can be extended to a set of allowable perturbations. In Section~\ref{sec:allowable} we show how to find an allowable perturbation. Section~\ref{sec:Crit} motivates and defines a recursive application of the algorithm provided in Section~\ref{sec:allowable}. The example case of assignment churning is covered in Section~\ref{sec:uav}, and Section~\ref{sec:conclusion} concludes.

\section{Problem Formulation}\label{sec:probForm}

Define a bipartite graph $\mathcal{G}=(\mathcal{V},\mathcal{E})$ along with a set of weights $\mathcal{W}$ over which the LAP is solved. Let the set of vertices $\mathcal{V} = \mathcal{V}_A \cup \mathcal{V}_T$ be such that $\mathcal{V}_A \cap \mathcal{V}_T = \emptyset$, with the edge set $\mathcal{E} \subseteq \mathcal{V}_A \times \mathcal{V}_T$, and edge weights $\mathcal{W} = \{ w_{ij} \in \mathbb{R} \mid (i,j) \in \mathcal{E} \}$. The vertices $\mathcal{V}_A$ and $\mathcal{V}_T$ represent the agents and tasks respectively. Without loss of generality we may assume there are not more tasks than agents, i.e., $|\mathcal{V}_A| \geq |\mathcal{V}_T|$. Also define the assignment as a set of binary decision variables $\Pi = \{ \pi_{ij} \in \{0,1\} \mid (i,j) \in \mathcal{E}\}$. The linear assignment problem can be formulated as

\begin{subequations}
\begin{align}
&  \min_\Pi  &&\sum_{(i,j) \in \mathcal{E}} \pi _{ij} w_{ij} \label{eq:LAP} \\
&  \;\;\textrm{s.t.}  &&\sum_{i \in \mathcal{V}_A} \pi_{ij} = 1,\phantom{ba} j \in \mathcal{V}_T, \label{eq:fullMatching}\\
&  &&\sum_{j \in \mathcal{V}_T} \pi_{ij} \leq 1,\phantom{ba} i \in \mathcal{V}_A. \label{eq:taskDominant}
\end{align}
\label{eq:lapOpt}
\end{subequations}

If $\pi_{ij} = 1$, then we say that that edge $(i,j)$ is assigned or vertex $i$ is assigned to vertex $j$. Constraints \eqref{eq:fullMatching} and \eqref{eq:taskDominant} ensure that every agent is assigned to either one or zero other tasks and that every task is assigned to one agent. The cost of an assignment $\Pi$ can be written as
\begin{equation}
f(\mathcal{W}, \Pi) = \sum_{(i,j) \in \mathcal{E}} \pi_{ij}w_{ij}.
\end{equation} 

There are many algorithms to solve the LAP, such as the Hungarian Algorithm\cite{kuhn1955hungarian}, the JV algorithm\cite{jonker1986improving}, or Bertsekas' auction algorithm\cite{bertsekas1988auction}. In this paper we assess the robustness of the optimal assignment, independent of the algorithm used to solve~\eqref{eq:lapOpt}. Define the mapping $LAP(\mathcal{G},\mathcal{W})$ which takes a weighted bipartite graph and returns the optimizer of~\eqref{eq:lapOpt}. If there are multiple equivalent optimal assignments, $LAP(\mathcal{G},\mathcal{W})$ returns the set of optimal assignments. In this paper we assume that the initial optimal assignment is unique, as the presence of degeneracy in the beginning significantly complicates the discussion. Degeneracy in the initial assignment will be studied in future work.

Let $\Delta := \{\delta_{ij} \mid (i,j)\in\mathcal{E}\}$ be a set of scalar edge weight perturbations. For convenience, we also define the addition of the set of perturbations to the set of edge weights
\begin{equation*}
\mathcal{W}+\Delta := \{w_{ij}+\delta_{ij} \mid (i,j)\in\mathcal{E} \}.
\end{equation*}
We now define robustness, and the concept of an \emph{allowable} perturbation. 

\begin{definition} \label{def:robustness} 
Define $\Delta$ to be an additive perturbation to the set of edge weights $\mathcal{W}$. If the optimal assignment $\Pi^*\in LAP(\mathcal{G},\mathcal{W})$ is invariant to the addition of $\Delta$, i.e., $\Pi^* \in LAP(\mathcal{G},\mathcal{W}+\Delta)$, then $\Delta$ is an allowable perturbation. Equivalently, $\Pi^*\in LAP(\mathcal{G},\mathcal{W})$ is robust to the perturbation $\Delta$.
\end{definition}

With this definition of robustness, the problem addressed in this paper can be stated as follows

\begin{problem*}
Given a bipartite graph $\mathcal{G}$, weights $\mathcal{W}$, with optimal assignment $\Pi^*\in LAP(\mathcal{G},\mathcal{W})$, characterize a set $\Lambda$ such that if $\Delta \in \Lambda$ then $\Delta$ is an allowable perturbation.
\end{problem*}

\section{Interval Bound}\label{sec:intervals}

We first show that if we have found an allowable perturbation $\Delta$, the linearity of the cost function allows us to construct an entire set of allowable perturbations based on $\Delta$.
\begin{theorem}
Let $\Pi^* \in LAP(\mathcal{G},\mathcal{W})$, and let $\Delta = \{\delta_{ij} \mid (i,j)\in\mathcal{E} \}$ be an allowable perturbation as in Definition~\ref{def:robustness}. Define another perturbation $\Delta' = \{\delta'_{ij} \mid (i,j)\in\mathcal{E} \}$ such that
\begin{align}
\delta'_{ij} \leq \delta_{ij} \; \forall \; (i,j) : \pi^*_{ij}=1,\label{eq:opt1}\\
\delta'_{ij} \geq \delta_{ij} \; \forall \; (i,j) : \pi^*_{ij}=0.\label{eq:opt2} 
\end{align} 
Then $\Delta'$ is an allowable perturbation. \label{thm:HalfSpace}
\end{theorem}

\begin{proof}
We will show that any perturbation $\Delta'$ which satisfies~\eqref{eq:opt1} and \eqref{eq:opt2} can be represented as the sum of the perturbation $\Delta$ and $\Delta''$, which are both allowable. By the linearity of the cost function then, we conclude that $\Delta'$ is allowable.

Let $\Delta''=\{\delta'_{ij}-\delta_{ij} \mid (i,j)\in\mathcal{E} \}$. By the linearity of the cost function~\eqref{eq:LAP},
\begin{align}
f(\mathcal{W}+\Delta',\Pi^*) = f(\mathcal{W}+\Delta,\Pi^*) + f(\Delta'',\Pi^*). \label{eq:linCost}
\end{align}

By assumption, $\Delta$ is an allowable perturbation. In other words, for $\Pi^*$ the optimizing assignment over the graph $\mathcal{G}$ with weights $\mathcal{W}$, $\Pi^*$ is an optimizing assignment for weights $\mathcal{W}+\Delta$. Additionally, by construction, $\delta ''_{ij}$ is non-positive for all assigned edges in $\Pi ^*$ and non-negative for all unassigned edges in $\Pi^*$. Therefore,
\begin{align*}
f(\Delta'',\Pi^*) \leq f(\Delta'',\Pi) \quad \forall \;\Pi
\end{align*}
In other words, no other assignment will have a experience/have a greater reduction in cost than $\Pi^*$. Therefore, as both terms on the right hand side of~\eqref{eq:linCost} are minimized by assignment $\Pi^*$, $\Pi^*$ is the minimizer over the graph with weights perturbed by $\Delta'$, i.e. $\Delta'$ is an allowable perturbation.
\end{proof}
Intuitively, Theorem~\ref{thm:HalfSpace} demonstrates that any perturbation which only reduces the weight of edges in the optimal assignment and increases the weight of edges not in the optimal assignment is an allowable perturbation. Using the theorem, any allowable perturbation $\Delta$ defines a set $\Lambda _{\Delta}$ of perturbations which the graph is also robust to.
\begin{example} 
Consider a scenario with $|\mathcal{V_A}|=|\mathcal{V_T}|=3$ and $\mathcal{E}=\mathcal{V}_A\times \mathcal{V}_T$. The weight set $\mathcal{W}$ is represented by a cost matrix $C$, where $w_{ij}$ is the $ij$th element. This is a convenient representation and will be used for the rest of the paper.
\begin{align*} C = 
\begin{bmatrix}
 91 & 33 & 15 \\
 5 & 86 & 92 \\
 85 & 9 & 42
\end{bmatrix}.
\end{align*}
The edges assigned in the optimal matching $\Pi^*$ over this cost matrix are $(2,1),(3,2),(1,3)$ with cost $f(\Pi^*,\mathcal{W})=29$. Trivially, the graph is robust to a perturbation of $0$ to each edge $\Delta := \{0 \mid (i,j) \in \mathcal{E} \}$. Using Theorem~\ref{thm:HalfSpace}, this set of perturbations can be interpreted as a set of intervals as shown in Table~\ref{tble:Trivial}.

\begin{table}[ht]
\begin{center}
\caption{Trivial Perturbation Intervals \label{tble:Trivial}}
\resizebox{0.35\textwidth}{!}{
\begin{tabular}{c  c  c }
$[0,\infty)$ & $[0,\infty)$ & $(-\infty,0]$ \\ 
$(-\infty,0]$ & $[0,\infty)$ & $[0,\infty)$ \\ 
$[0,\infty)$ & $(-\infty,0]$ & $[0,\infty)$
\end{tabular}
}
\end{center}
\end{table}

By Theorem~\ref{thm:HalfSpace}, any other weight perturbation $\Delta'$ whose elements are within the intervals shown in Table~\ref{tble:Trivial} is allowable.
\end{example}

\section{Construction of Allowable Perturbation}\label{sec:allowable}

In the previous section we showed that an allowable perturbation can be interpreted as a bound on a set of allowable perturbations. However, this does not provide us with a method of constructing an allowable perturbation. In order to construct an allowable perturbation, we expand on the \emph{element-wise} results derived in~\cite{KINDERVATER198576} where a similar question of robustness to perturbations in a single edge weight are considered. For an edge $(a,b)\in\mathcal{E}$, define the element-wise sensitivity $s_{ab}$ as
\begin{equation}
s_{ab} = 
\begin{cases}
	f(\mathcal{W},\Pi^*) - f(\mathcal{W},\Pi_{ab}) & \textrm{if } \pi^*_{ab} = 0\\
	f(\mathcal{W},\Pi_{ab}) - f(\mathcal{W},\Pi^*) & \textrm{if } \pi^*_{ab} = 1
\end{cases}
\label{eq:sensMod}
\end{equation}
where the assignment $\Pi_{ab}$ is defined to be the optimizer of~\ref{eq:LAP} with the additional constraint~\eqref{eq:ForBconstraint}
\begin{subequations}
\begin{align}
&  \min_\Pi  &&\sum_{(i,j) \in \mathcal{E}} \pi _{ij} w_{ij} \\
&  \;\;\textrm{s.t.}  &&\sum_{(i,j) \in \mathcal{E}} \pi_{ij} = 1,\\
&  &&\sum_{(i,j) \in \mathcal{E}} \pi_{ij} \leq 1,\\
&  &&\pi_{ab} \neq \pi^*_{ab} . \label{eq:ForBconstraint}
\end{align}
\label{eq:kinderModified}
\end{subequations}
The new constraint~\eqref{eq:ForBconstraint} either forces the assignment of an edge $(a,b)$, if the edge was not included in the original optimal assignment, or it blocks the edge $(a,b)$ if it was part of the original optimal assignment. To understand the effect this has, consider constructing a set of weights $\bar{\mathcal{W}}$ that is the same as $\mathcal{W}$ except perturb the weight $w_{ab}$ by its sensitivity $s_{ab}$. By construction, the assignment costs $f(\bar{\mathcal{W}},\Pi_{ab})$ and $f(\bar{\mathcal{W}},\Pi^*)$ would then be equal. Therefore each $s_{ab}$ represents a bound on the allowable element-wise perturbation of the LAP over the graph $\mathcal{G}$ with weights $\mathcal{W}$. Here we use \emph{element-wise} to indicate only one edge weight is perturbed and all other edge weights are held fixed.

Clearly, the set of element-wise sensitivities do not form an allowable perturbation when taken together. To illustrate, consider the costs of the optimal solution $\Pi^*$ and some other solution $\bar{\Pi} \neq \Pi^*$ over the graph with weights perturbed by the set of all sensitivities, as defined in~\eqref{eq:sensMod}, $\Delta_s = \{s_{ij} \mid (i,j)\in\mathcal{E}\}$.
\begin{align}
f(\mathcal{W}+\Delta_s,\bar{\Pi}) - f(\mathcal{W}+\Delta_s,\Pi^*) &=  f(\mathcal{W},\bar{\Pi}-\Pi^*) \nonumber\\
 &\qquad+ f(\Delta_s,\bar{\Pi}-\Pi^*) \label{eq:pertCostDif}
\end{align}
In order for $\Delta_s$ to be an allowable perturbation, the cost difference~\eqref{eq:pertCostDif} must be non-negative for all assignments $\bar{\Pi}$. However, for each edge $(i,j)\in \bar{\Pi}$, $s_{ij}$ may be as large as $f(\mathcal{W},\Pi^* - \bar{\Pi})$, as in~\eqref{eq:sensMod}. Similarly, for each edge in $(i,j)\in\Pi^*$, $s_{ij}$ may be as large as $f(\mathcal{W},\bar{\Pi}-\Pi^*)$. Any of these perturbations individually may be sufficient to bring the cost difference to zero and combinations of the individual perturbation can make \eqref{eq:pertCostDif} negative. We therefore modify the sensitivity values to form an allowable perturbation for the graph $\mathcal{G}$.

\begin{prop} \label{prop:LowerBound}
Let $\Pi^*$ be the optimal assignment over a bipartite graph $\mathcal{G}$ with weights $\mathcal{W}$ and let the set of values $\mathcal{S}=\{s_{ij} \mid (i,j) \in \mathcal{E} \}$ be the element-wise sensitivities as defined in~\eqref{eq:sensMod}. For $N$ the maximal number of assigned edges in an assignment over $\mathcal{G}$, the perturbation
\begin{align}
\Delta := \{\delta_{ij} =\frac{s_{ij}}{2N} \mid (i,j)\in\mathcal{E} \} \label{eq:allowable2N}
\end{align}
is an allowable perturbation. 
\end{prop}
\begin{proof}
Define a bipartite graph $\mathcal{G} = (\mathcal{V},\mathcal{E})$ and a set of weights $\mathcal{W}$, with the optimal assignment $\Pi^*\in LAP(\mathcal{G},\mathcal{W})$. For an assignment $\Pi' \neq \Pi^*$, the perturbed cost difference is
\begin{align*}
f(\mathcal{W}+\Delta,\Pi') - f(\mathcal{W}+\Delta,\Pi^*) &= f(\mathcal{W},\Pi'-\Pi^*) \\
&\qquad+ f(\Delta,\Pi'-\Pi^*).
\end{align*}
If this difference is non-negative for all assignments $\Pi'$, then $\Pi^*$ is the optimizer of the graph $\mathcal{G}$ with weights $\mathcal{W}+\Delta$, and $\Delta$ is allowable. Since $\Pi^*$ is the optimizer with the weights $\mathcal{W}$, $f(\mathcal{W},\Pi'-\Pi^*)$ must be non-negative for all assignments $\Pi'$. If $f(\Delta,\Pi'-\Pi^*)$ is non-negative as well then the proof is trivial. However, in the case where $f(\Delta,\Pi'-\Pi^*)$ is negative, if we can show that 
\begin{align}
f(\Delta,\Pi^*-\Pi') \leq f(\mathcal{W},\Pi'-\Pi^*) \label{eq:proofGoal}
\end{align}
then $\Delta$ is an allowable perturbation.

First we expand the term on the left of~\eqref{eq:proofGoal} into the sum of individual perturbations
\begin{equation*}
f(\Delta,\Pi^*-\Pi') = \sum_{(i,j)\in\mathcal{E}} \pi^*_{ij}\delta_{ij} - \pi'_{ij}\delta_{ij}. 
\end{equation*}
Weight perturbations of edges assigned in both $\Pi'$ and $\Pi^*$ will not effect the cost difference, so let $\bar{\mathcal{E}} = \{(i,j) \mid \pi_{ij}\neq\pi^*_{ij}\; \forall \;(i,j) \in\mathcal{E}\}$ be the set of edges which are assigned in $\Pi'$ or $\Pi^*$ but not both. The cost difference can now be written as
\begin{align*}
f(\Delta,\Pi^*-\Pi') &= \sum_{(i,j)\in\bar{\mathcal{E}}} \pi^*_{ij}\delta_{ij} - \pi'_{ij}\delta_{ij}.
\end{align*}

Recall that for each $\delta_{ij}=\frac{s_{ij}}{2N}$, element-wise sensitivity $s_{ij}$ is constructed using the difference between the optimal assignment $\Pi^*$ and the lowest cost assignment which forces or blocks $(i,j)$, $\Pi_{ij}$, as defined in~\eqref{eq:kinderModified}. Using this, we can bound each perturbation $\delta_{ij}\in\bar{\mathcal{E}}$
\begin{align*}
\sum_{(i,j)\in\bar{\mathcal{E}}} \pi^*_{ij}\frac{s_{ij}}{2N} &= \sum_{(i,j)\in\bar{\mathcal{E}}} \pi^*_{ij}\frac{f(\mathcal{W},\Pi_{ij}) - f(\mathcal{W},\Pi^*)}{2N},\\
&\leq \sum_{(i,j)\in\bar{\mathcal{E}}} \pi^*_{ij}\frac{f(\mathcal{W},\Pi'-\Pi^*)}{2N},
\end{align*}
where the bound on the perturbations for the edges assigned in $\Pi'$ are the same except for a negative sign in the numerator, as in~\eqref{eq:sensMod}. Substituting these upper bounds
\begin{align*}
&f(\Delta,\Pi^*-\Pi') = \sum_{(i,j)\in\bar{\mathcal{E}}} \pi^*_{ij}\delta_{ij} - \sum_{(i,j)\in\bar{\mathcal{E}}} \pi'_{ij}\delta_{ij},\\
&\leq \sum_{(i,j)\in\bar{\mathcal{E}}} \pi^*_{ij}\frac{f(\mathcal{W},\Pi'-\Pi^*)}{2N} - \pi'_{ij}\frac{f(\mathcal{W},\Pi^*-\Pi')}{2N},\\
&\leq \sum_{(i,j)\in\bar{\mathcal{E}}} \pi^*_{ij}\frac{f(\mathcal{W},\Pi'-\Pi^*)}{2N} + \pi'_{ij}\frac{f(\mathcal{W},\Pi'-\Pi^*)}{2N},\\
&\leq N\frac{f(\mathcal{W},\Pi'-\Pi^*)}{2N} + N\frac{f(\mathcal{W},\Pi'-\Pi^*)}{2N},\\
&= f(\mathcal{W},\Pi'-\Pi^*).
\end{align*}
Where we removed the summations by noting that $N$ is the largest number of assigned edges in an assignment. Therefore, since the weight perturbation for any solution $\Pi'$ is less than $f(\mathcal{W},\Pi') - f(\mathcal{W},\Pi^*)$, $\Pi^*$ is an optimizer with weight $\mathcal{W}+\Delta$, so $\Delta$ is allowable.
\end{proof}
Using Proposition~\ref{prop:LowerBound} we can construct an allowable perturbation $\Delta$, and using Theorem~\ref{thm:HalfSpace} this can be extended into a set of allowable perturbations, as desired in the problem statement in Section~\ref{sec:probForm}. 

\begin{example}
Recall the cost matrix from the previous example. Using \eqref{eq:kinderModified}, the matrix of element-wise sensitivities as defined in~\ref{eq:sensMod} is,
\begin{align*} S = 
\begin{bmatrix}
 -163 & -51 & 51 \\
 157 & -157 & -163 \\
 -157 & 51 & -51
\end{bmatrix}.
\end{align*}
Dividing the sensitivities by $2N=6$, and converting into intervals as before, the set of allowable perturbations is shown in Table~\ref{tble:2ndtoLast}
\begin{table}[ht]
\begin{center}
\caption{Allowable Perturbation Intervals from Proposition~\ref{prop:LowerBound} \label{tble:2ndtoLast}}
\resizebox{0.40\textwidth}{!}{
\begin{tabular}{c  c  c }
$[-27.1,\infty)$ & $[-8.5,\infty)$ & $(-\infty,8.5]$ \\ 
$(-\infty,26.1]$ & $[-26.1,\infty)$ & $[-27.1,\infty)$ \\ 
$[-26.1,\infty)$ & $(-\infty,8.5]$ & $[-8.5,\infty)$ \\ 
\end{tabular}
}
\end{center}
\end{table}
\end{example}

\section{Critical Perturbation}\label{sec:Crit}

So far we have shown how to construct a set of allowable perturbations, using a particular allowable perturbation $\Delta$. Therefore, we would like to maximize (or minimize) the bounding weights in $\Delta$, in order to obtain the largest set of tolerances for application. To this end, we define a \emph{critical perturbation}.

\begin{definition}\label{def:critPert}
Let $\Pi^* \in LAP(\mathcal{G},\mathcal{W})$, and let $\Delta = \{\delta_{ij} \mid (i,j)\in\mathcal{E} \}$ be an allowable perturbation as in Definition~\ref{def:robustness}. Define another perturbation $\Delta' = \{\delta'_{ij} \mid (i,j)\in\mathcal{E} \}$ satisfying
\begin{align*}
\delta'_{ij} = \delta_{ij} + \epsilon_{ij}\; \forall \; (i,j) : \pi^*_{ij}=1,\\
\delta'_{ij} = \delta_{ij} - \epsilon_{ij}\; \forall \; (i,j) : \pi^*_{ij}=0. 
\end{align*} 
The allowable perturbation $\Delta$ is \emph{critical} if $\Delta'$ is not allowable for \emph{any} $\epsilon_{ij} > 0$.
\end{definition}

Intuitively, a critical perturbation $\Delta$ is allowable buton the boundary of not being so.

With this definition, we may note an interesting connection to the element-wise sensitivity discussed in the previous section. The element-wise sensitivity is the amount that a weight can be perturbed, with all other edges fixed, without changing the optimal solution. Therefore, equivalent to Definition~\ref{def:critPert}, a perturbation $\Delta$ is critical if and only if the element-wise sensitivities of $\mathcal{W}+\Delta$ are all zero. To check if the allowable perturbation $\Delta$ returned from Proposition~\ref{prop:LowerBound} is critical we can simply recompute the set of element-wise sensitivities of $\mathcal{W}+\Delta$. If they are zero, then the perturbation is critical. However, if they are not zero, then we can divide them by $2N$ as in Proposition~\ref{prop:LowerBound} construct a new allowable perturbation. The recursive algorithm is described in pseudo-code in Algorithm~\ref{alg:Recursive}, where the function "sensitivities" returns the set of element-wise sensitivities $s_{ij}$ as defined in~\eqref{eq:sensMod}.
\begin{algorithm}
\KwData{$\mathcal{G},\mathcal{W},\Pi^*,N$}

$S\leftarrow$ sensitivities($\mathcal{G},\mathcal{W},\Pi^*$);\\
$\Delta\leftarrow \frac{1}{2N}S$;

$S\leftarrow$ sensitivities($\mathcal{G},\mathcal{W}+\Delta,\Pi^*$);\\

\While{$S \neq \{0 \mid (i,j)\in\mathcal{E}\}$}{
	$\Delta\leftarrow \Delta + \frac{S}{2N}$;\\
	$S\leftarrow$ sensitivities($\mathcal{G},\mathcal{W}+\Delta,\Pi^*$);
}
\caption{Constructing a critical perturbation over $\mathcal{G}$ with respect to $\mathcal{W}$ and $\Pi^*$. \label{alg:Recursive}}
\end{algorithm}

\begin{prop}\label{prop:recursive}
For a bipartite graph $\mathcal{G} = (\mathcal{V},\mathcal{E})$ with weights $\mathcal{W}$ and optimal assignment $\Pi^*\in LAP(\mathcal{G},\mathcal{W})$, Algorithm~\ref{alg:Recursive} converges and the resulting perturbation $\Delta$ is critical in the sense of definition~\ref{def:critPert}.
\end{prop}

\begin{proof}
We first prove that the sum of the perturbations $\Delta$ converges by showing that the sequence is monotonic and bounded, and thus also show that the element-wise sensitivities converge to zero. The sensitivities converging to zero thus proves that the perturbation $\Delta$ is critical.

Let $\Delta^k = \{\delta^k_{ij} \mid (i,j)\in\mathcal{E}\}$ be the perturbation at the $k$-th iteration of the while loop in lines $4-7$ of Algorithm~\ref{alg:Recursive}, and likewise let $s^k_{ij}$ be the element-wise sensitivity of the edge $(i,j)$, as defined in~\eqref{eq:sensMod}, at the $k$-th iteration. At iteration $k+1$, the perturbations increment according to
\begin{align}
\delta^{k+1}_{ij} = \delta^{k}_{ij} + \frac{s^k_{ij}}{2N}.
\end{align}
For all edges $(i,j)$ such that $\pi^*_{ij}=1$, $s^k_{ij}\geq0$. Likewise for all edges $(i,j)$ such that $\pi^*_{ij}=0$, $s^k_{ij}\leq0$. Therefore, for each edge $(i,j)$, the sequence $\delta^k_{ij}$ is \emph{monotonic} non-decreasing/non-increasing for edges assigned/not-assigned in the optimizer.

Each perturbation $\delta_{ij}=\sum_k \frac{s^k_{ij}}{2N}$ is bounded by its element-wise sensitivity over the original graph $\mathcal{G} = (\mathcal{V},\mathcal{E})$ with weights $\mathcal{W}$, i.e.
\begin{align*}
|\delta^{k}_{ij}| \leq |s_{ij}| \; \forall\; k.
\end{align*}
If an edge perturbation $\delta_{ij}$ was larger than the sensitivity $s_{ij}$, then the perturbation $\Delta$ would not be allowable. Since $\Delta ^{k}$ is allowable at each iteration $k$ by construction, the weight perturbations perturbation $\delta_{ij}$ must be bounded. The sequence of edge perturbations $\delta^{k}_{ij}$ is both monotonic and the sum is bounded, so it must converge.

By the convergence of each sequence $\delta^{k}_{ij}$, the sequence $s^k_{ij}$ must converge to zero. In other words, as $k\rightarrow \infty$, the perturbed weights $\mathcal{W} + \Delta^k$ have element-wise sensitivities of $0$ for all edges. Therefore, as $k\rightarrow \infty$, $\Delta^k$ is a critical perturbation.

\end{proof}

\begin{example}
Recall the cost matrix from the previous example. Using Algorithm~\ref{alg:Recursive} to obtain an allowable perturbation, and Theorem~\ref{thm:HalfSpace} to convert the allowable perturbation into a set of intervals, the results are shown in Table \ref{tble:Last}.
\begin{table}[ht]
\begin{center}
\caption{Allowable Perturbations from Algorithm~\ref{alg:Recursive} \label{tble:Last}}
\resizebox{0.40\textwidth}{!}{
\begin{tabular}{c  c  c }
$[-37,\infty)$ & $[-8.5,\infty)$ & $(-\infty,8.5]$ \\ 
$(-\infty,35]$ & $[-35,\infty)$ & $[-37,\infty)$ \\ 
$[-35,\infty)$ & $(-\infty,8.5]$ & $[-8.5,\infty)$ 
\end{tabular}
}
\end{center}
\end{table}
\end{example}

\section{Multi-Vehicle Guidance}\label{sec:uav}

In this section we apply the assignment sensitivity analysis to a multi-vehicle guidance problem. Multi-vehicle guidance may involve several vehicles cooperating to achieve more complex tasks, necessitating task assignment among the agents. The assignment may be sensitive to errors in the input data, leading to potentially suboptimal assignments during the completion of the task. A naive approach is to reassign tasks as new data is available under the assumption that the most recent measurements are the most reliable. However, continuously re-solving the assignment problem may be costly. Furthermore, reassignment based on real-time measurements while agents are completing their assigned tasks may lead to repeated switching of tasks referred to as churning behavior\cite{bertuccelli2009real}\cite{alighanbari2004filter}. To mitigate the effects of assignment churning, a new objective function focusing on robustness and penalizing the switching of assignments was used in \cite{bertuccelli2009real}. However, this may compromise the optimality of the assignment, as the objective function is modified. The assignment invariant perturbation intervals derived in this paper can be used to guarantee the optimality of an assignment, preventing the need for reassignment, and avoiding churning.

Consider a set of vehicle positions $\{x_i\}_{i=0}^N$, $x_i\in\mathbb{R}^2$, and target destinations $\{y_i\}_{i=0}^N$, $y_i\in\mathbb{R}^2$, shown in Fig.~\ref{fig:churning} as dots and crosses. Let the edge weights be
\begin{align*}
w^{(k)}_{ij} = d(x^{(k)}_i,y^{(k)}_j) + \epsilon^{(k)}_{ij}
\end{align*}
where $d(\cdot)$ is a distance function and $\epsilon^{(k)}_{ij}$, with $|\epsilon^{(k)}_{ij}|\leq\bar{\epsilon}^{(k)}_{ij}$ the bounded error of the distance measurement at step $k$. The set of measurements at each time step may be naively used to construct the assignment at time $k$, $\Pi_k$, subject to a new realization of the noise $\epsilon^{(k)}_{ij}$. This can lead to reassignment while the agents are traveling, or churning, as shown in Fig.~\ref{fig:churning}.

\begin{figure}[thpb]
  \centering
  \includegraphics[scale=0.5]{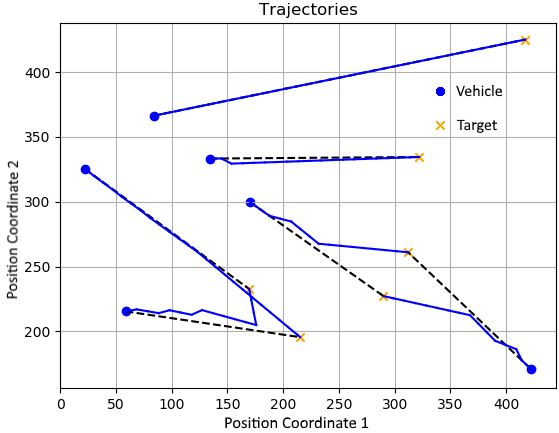}
  \caption{Assignment Churning \label{fig:churning}}
\end{figure}

In Fig.~\ref{fig:churning}, the the solid blue lines show the trajectories of agents that may begin with the optimal assignment, but without ground truth knowledge continue to reassign on new data. Then, as the agents draw nearer to the targets, the assignments settle. Clearly, the total distance traveled is significantly increased due to the deviations from the optimal path, shown as a dashed line. Given that the agents are traveling the shortest paths to their assigned destinations, if the assignment was optimal at any point then reassignment can only increase the total traveled distance.

The bounds for allowable perturbations as described in Section~\ref{sec:allowable} or Section~\ref{sec:Crit} can be used to avoid churning. Let $\Delta = \{\delta_{ij} \mid (i,j)\in\mathcal{E}\}$ be the allowable perturbation as described in Section~\ref{sec:Crit}, and recall that $\bar{\epsilon}^{(k)}_{ij}$ bounds the error for each edge weight at each step. If, at some reassignment step $k$, the error bounds satisfy
\begin{align}
\bar{\epsilon}^{(k)}_{ij} \leq \delta^{(k)}_{ij} \;\forall\; \pi^{(k)}_{ij} = 1, \label{eq:optCond1}\\
\delta^{(k)}_{ij} \leq -\bar{\epsilon}^{(k)}_{ij}\;\forall\; \pi^{(k)}_{ij} = 0, \label{eq:optCond2}
\end{align}
then the assignment over the measurements is the optimal assignment over the ground truth states. This is easily shown by noting that any realization of the bounded measurement errors $\epsilon^{(k)}_{ij}$ will form an allowable perturbation by Theorem~\ref{thm:HalfSpace}, and the optimal assignment is invariant under these errors. Therefore, the assignment $\Pi_k$ made over the measurements must be equal to the optimal assignment $\Pi^*_k$ made on the noiseless measurements. After this time step, reassignment can only increase the total traveled distance. This avoids churning and any computation involved with constructing the new costs and finding the new assignments.

\section{Concluding Remarks}\label{sec:conclusion}

In this paper the robustness of the linear assignment problem (LAP) to a set of perturbations is analyzed, and applied to prevent assignment churning without changing the objective function. We first showed that a perturbation which does not change the optimal assignment can be interpreted as the bound on a set of allowable perturbations, and then constructed an allowable perturbation based on the element-wise sensitivities. We also proved that this method can be used recursively to construct critical perturbation. These element-wise bounds can be used to ensure LAP optimality under uncertainty, or to prevent unnecessary resolving of the LAP in online scenarios.

In future work, a closed form solution for the recursive algorithm will be investigated. Similar robustness for various other assignment problem formulations are to be investigated. Another research direction is the interpretation of these results as a strategy to deploy agents informed by the robustness. This would ensure that reassignments would be unlikely, even under uncertainty. In a similar vein, a definition of sensitivity which extends to sub-optimal assignments, and an algorithm to explore their utility, is a subject of interest.

\bibliographystyle{IEEEtran}
\bibliography{IEEEabrv,acc}

\end{document}